\documentclass[10pt]{amsart}
\usepackage{enumerate,amsmath,amssymb,latexsym,
amsfonts, amsthm, amscd, MnSymbol}

\theoremstyle{plain}

\newtheorem{theorem}{Theorem}

\numberwithin{equation}{section}

\newcommand{\ra}{\rightarrow}

\newcommand{\mdot}{\,\begin{picture}(-1,-1)(-1,-1)\circle*{2}\end{picture}\ }

\def\barroman#1{\sbox0{#1}\dimen0=\dimexpr\wd0+1pt\relax
  \makebox[\dimen0]{\rlap{\vrule width\dimen0 height 0.06ex depth 0.06ex}%
    \rlap{\vrule width\dimen0 height\dimexpr\ht0+0.03ex\relax 
            depth\dimexpr-\ht0+0.09ex\relax}%
    \kern.5pt#1\kern.5pt}}

\addtocounter{section}{-1}

\begin{document}

\title {Painlev\'e II transcendents and their squares}

\date{}

\author[P.L. Robinson]{P.L. Robinson}

\address{Department of Mathematics \\ University of Florida \\ Gainesville FL 32611  USA }

\email[]{paulr@ufl.edu}

 \keywords{}

\begin{abstract}

We examine the relationship between the homogeneous second Painlev\'e equation and equation ${\bf XX}$ from the master list recorded by Ince. 

\end{abstract}

\maketitle

\medbreak

\section*{Introduction} 

\medbreak 

Prompted by findings of Kowalevski in her analysis of gyroscopic motion, Picard initiated a study of second-order ordinary differential equations of the form 
$$\frac{{\rm d}^2 w}{{\rm d} z^2} = F\Big(z, w, \frac{{\rm d} w}{{\rm d} z}\Big)$$
in which the right side is analytic in $z$ but rational in $w$ and ${\rm d} w/{\rm d} z$. Painlev\'e and his student Gambier classified all such ODEs possessing the property that their solutions have only poles among their movable singularities. Their classification resulted in 50 canonical forms: these are listed explicitly in [Ince] and each is traditionally labelled by a Roman numeral according to its position in this list. Six of these 50 equations are separated from this list, freshly labelled ${\bf PI}$ through ${\bf PVI}$ and called the {\it Painlev\'e equations}, their solutions being {\it Painlev\'e transcendents}; solutions to all 50 may be expressed in terms of solutions to these six along with solutions to `classical' ODEs (including linear equations and those that define elliptic functions). 

\medbreak 

In fact, one of the Painlev\'e equations is very intimately related to another in the list of 50. The second Painlev\'e equation ${\bf PII}$ has the form 
$$\frac{{\rm d}^2 w}{{\rm d} z^2} = 2 w^3 + z w + \alpha$$
in which $\alpha$ is a parameter. The special case in which $\alpha = 0$ is called homogeneous; we shall label it ${\bf PII}_0$. The twentieth equation {\bf XX} in the list of 50 has the form 
$$\frac{{\rm d}^2 w}{{\rm d} z^2} = \frac{1}{2 w}\Big( \frac{{\rm d} w}{{\rm d} z}\Big)^2 + 4 w^2 + 2 z w$$
in which the quotient on the right is to be understood as a limit when appropriate. It is asserted on page 337 of [Ince] that {\bf XX} is equivalent to ${\bf PII}_0$ by squaring. 

\medbreak 

Here we examine more closely certain aspects of the relationship between {\bf XX} and ${\bf PII}_0$. To be specific, we restrict attention to {\bf XX} and ${\bf PII}_0$ as {\it real} equations with real solutions. In this context, squares of nowhere-zero solutions to ${\bf PII}_0$ satisfy ${\bf XX}$ while positive square-roots of strictly positive solutions to {\bf XX} satisfy ${\bf PII}_0$. When solutions are allowed to acquire (isolated) zeros we find that there is a sudden change in behaviour, which we analyze in detail. Our examination brings out a significant property of equation {\bf XX}. The presence of $w$ in the denominator on the right side of {\bf XX} means that the standard (local) existence-uniqueness theorem for a second-order ODE does not apply to {\bf XX} when the initial data involve a zero of the solution. We observe that further differentiation leads to a third-order ODE in which the right side is polynomial in all variables; accordingly, the standard existence-uniqueness theorem for a {\it third}-order ODE applies to this equation. As a direct consequence, a solution to {\bf XX} that vanishes at a point is uniquely determined by the value of its {\it second} derivative there. 

\medbreak 

\section*{{\bf XX} and an associated third-order ODE} 

\medbreak 

As we mentioned in the Introduction, we shall regard {\bf XX} as a real ordinary differential equation with real solutions. Thus, we shall write this equation as 
\begin{equation} \label{XX} 
\overset{\mdot \mdot}{S} = \frac{\overset{\mdot}{S}^2}{2 S} + 4 S^2 + 2 t S, \tag{{$\bf XX$}}
\end{equation}
where a superior dot $^{\mdot}$ signifies the derivative and where the ratio $\overset{\mdot}{S}^2/{2 S}$ is to be understood as a limit when appropriate. It follows that the derivative of a solution vanishes wherever the solution itself vanishes: if $S(a) = 0$ then automatically $\overset{\mdot}{S}(a) = 0$ also; this has consequences, as we shall see. 

\medbreak 

Notice that {\bf XX} has the form 
$$\overset{\mdot \mdot}{S} = F(t, S, \overset{\mdot}{S})$$
in which the right side is rational, with $S$ in the denominator. In consequence of this, the standard (local) existence-uniqueness theorem for second-order ODEs applies to {\bf XX} away from zeros: there exists a unique solution $S$ to {\bf XX} for which $S(a) \ne 0$ and $\overset{\mdot}{S}(a)$ have specified values. The standard existence-uniqueness theorem fails when the initial data involve a zero of the solution: indeed, $S(a) = 0$ entails $\overset{\mdot}{S}(a) = 0$ as noted above; were the standard theorem to apply, it would force $S$ to vanish identically on its interval domain. We analyze further the case of an isolated zero below. 

\medbreak 

It follows at once from {\bf XX} that each solution $S$ is thrice-differentiable away from its zeros: calculation of the third derivative starts conveniently from the reformulation
$$2 S \overset{\mdot \mdot}{S} = \overset{\mdot}{S}^2 + 8 S^3 + 4 t S^2;$$
after differentiation, $2 \overset{\mdot}{S} \overset{\mdot \mdot}{S}$ falls from each side so that
$$2 S \: \overset{\mdot \mdot \mdot}{S} = 24 S^2 \overset{\mdot}{S} + 8 t S \overset{\mdot}{S} + 4 S^2$$
and therefore 
$$\overset{\mdot \mdot \mdot}{S} = 12 S \overset{\mdot}{S} + 4 t \overset{\mdot}{S} + 2 S.$$
Now let the solution $S$ to {\bf XX} have an isolated zero at $a$. The understanding that the ratio on the right side of {\bf XX} is defined as a limit ensures that $\overset{\mdot \mdot}{S}$ is continuous at $a$. As we let $(a \ne ) \; t \ra a$ in the equation
$$\overset{\mdot \mdot \mdot}{S}(t) = 12 S(t) \overset{\mdot}{S}(t) + 4 t \overset{\mdot}{S}(t) + 2 S(t)$$
both $S(t) \ra S(a) = 0$ and $\overset{\mdot}{S}(t) \ra \overset{\mdot}{S}(a) = 0$ so that $\overset{\mdot \mdot \mdot}{S}(t) \ra 0$ also. We deduce that $S$ is also thrice-differentiable at $a$ with $\overset{\mdot \mdot \mdot}{S}(a) = 0$, by an application of the mean value theorem to the continuous function $\overset{\mdot \mdot}{S}$. 

\medbreak 

We may record the result of our recent deliberations as the following theorem; in its statement, we assume that the zeros of $S$ are isolated. 

\medbreak 

\begin{theorem} \label{X'}
If $S$ is a solution to {\bf XX} then $S$ satisfies the third-order equation 
\begin{equation} \label{XX'} 
\overset{\mdot \mdot \mdot}{S} = 12 S \overset{\mdot}{S} + 4 t \overset{\mdot}{S} + 2 S. \tag{{$\bf XX'$}}
\end{equation}
\end{theorem} 

\qed 

\medbreak 

Observe that equation ${\bf XX'}$ has the form 
$$\overset{\mdot \mdot \mdot}{S} = G(t, S, \overset{\mdot}{S}, \overset{\mdot \mdot}{S})$$
in which the right side is polynomial in all variables (and $\overset{\mdot \mdot}{S}$ is incidentally absent). The standard (local) existence-uniqueness theorem for a third-order ODE thus applies: there exists a unique solution to ${\bf XX'}$ having specified values of $S(a), \; \overset{\mdot}{S}(a)$ and $\overset{\mdot \mdot}{S}(a)$.  

\medbreak 

This has an immediate application to {\bf XX} itself. 

\medbreak 

\begin{theorem} \label{ne}
Let $S$ be a solution to {\bf XX}. If $S$ has an isolated zero at $a$ then $\overset{\mdot \mdot}{S}(a) \ne 0.$
\end{theorem} 

\begin{proof} 
According to Theorem \ref{X'}, $S$ is also a solution to ${\bf XX}'$. As we have seen, if $S(a) = 0$ then $\overset{\mdot}{S}(a) = 0$ automatically. If also $\overset{\mdot \mdot}{S}(a) = 0$ then the standard uniqueness theorem for solutions to the third-order equation ${\bf XX'}$ forces $S = 0$ and so prevents the zero at $a$ from being isolated. 
\end{proof}

\medbreak 

\section*{{\bf XX} in relation to homogeneous PII }

\medbreak 

Now we explore the relationship between ${\bf XX}$ and the homogeneous second Painlev\'e equation, which we record as 
\begin{equation} \label{PII} 
\overset{\mdot \mdot}{s} = 2 s^3 + t s \tag{${\bf PII}_0$}
\end{equation} 
and view in real terms. Throughout what follows, the intention is that lower case $s$ should suggest a solution to ${\bf PII}_0$ while upper case $S$ should suggest a solution to ${\bf XX}$. 

\medbreak 

Let $s$ be a nowhere-zero solution to ${\bf PII}_0$ and define $S := s^2$. Then $\overset{\mdot}{S} = 2 s \overset{\mdot}{s}$ and 
$$\overset{\mdot \mdot}{S} = 2 s \overset{\mdot \mdot}{s} + 2 \overset{\mdot}{s}^2$$
so that $\overset{\mdot}{s} = \overset{\mdot}{S}/2 s$ and 
$$\overset{\mdot \mdot}{S} = 2 s (2 s^3 + t s) + 2 \Big(\frac{\overset{\mdot}{S}}{2 s} \Big)^2 = 4 s^4 + 2 t s^2 + \frac{\overset{\mdot}{S}^2}{2 s^2}$$
whence 
$$\overset{\mdot \mdot}{S} = 4 S^2 + 2 t S + \frac{\overset{\mdot}{S}^2}{2 S}$$
which proves that $S$ is a solution to ${\bf XX}$. In the opposite direction, let $S$ be a strictly positive solution to ${\bf XX}$ and define $s := \sqrt S$ to be its positive square-root. A similar direct calculation using $\overset{\mdot}{S} = 2 s \overset{\mdot}{s}$ and the fact that $S$ satisfies ${\bf XX}$ shows that 
$$2 s \overset{\mdot \mdot}{s} = \overset{\mdot \mdot}{S} - 2 \overset{\mdot}{s}^2 =  \frac{\overset{\mdot}{S}^2}{2 S} + 4 S^2 + 2 t S - 2 \overset{\mdot}{s}^2 = 4 S^2 + 2 t S = 4 s^4 + 2 t s^2$$
so by cancellation 
$$\overset{\mdot \mdot}{s} = 2 s^3 + t s$$
and $s$ is a solution to ${\bf PII}_0$. 

\medbreak 

\begin{theorem} \label{nozero}
If $s$ is a nowhere-zero solution to ${\bf PII}_0$ then $s^2$ is a solution to ${\bf XX}$. If $S$ is a strictly positive solution to ${\bf XX}$ then $\sqrt S$ is a solution to ${\bf PII}_0.$ 
\end{theorem} 

\qed

\medbreak 

The presence of zeros introduces complications. As in the previous section, we take zeros to be isolated; more precisely, we consider a function (an $s$ or an $S$ as the case may be) that is defined on an open interval $I$ and vanishes at precisely one point $a \in I$. 

\medbreak 

\begin{theorem} \label{square}
If $s$ satisfies ${\bf PII}_0$ on $I$ and is zero only at $a \in I$ then $s^2$ satisfies ${\bf XX}$ on $I$. 
\end{theorem} 

\begin{proof} 
 Theorem \ref{nozero} guarantees that the twice-differentiable function $S : = s^2$ satisfies ${\bf XX}$ on $I \setminus \{a\}$; we must examine its behaviour at $a$. Note that  
$$4 S(a)^2 + 2 a S(a) = 0$$
and 
$$\overset{\mdot \mdot}{S}(a) = 2 s(a) \overset{\mdot \mdot}{s}(a) + 2 \overset{\mdot}{s}(a)^2 = 2 \overset{\mdot}{s}(a)^2$$
because $s$ vanishes at $a$. Note further that if $I \ni t \ne a$ then 
$$\frac{\overset{\mdot}{S}(t)^2}{2 S(t)} = \frac{(2 s(t) \overset{\mdot}{s}(t))^2}{2 s(t)^2} = 2 \overset{\mdot}{s}(t)^2$$ 
which converges to $2 \overset{\mdot}{s}(a)^2$ as $t \ra a$. We conclude that $S$ satisfies ${\bf XX}$ at $a$ too.  
\end{proof} 

\medbreak 

Thus squaring yields no surprises. The taking of square-roots is more interesting. 

\medbreak 

We begin with a negative result. 

\medbreak 

\begin{theorem} \label{notroot}
If $S$ satisfies ${\bf XX}$ on $I$ and is strictly positive except for a zero at $a \in I$ then $\sqrt S$ does not satisfy ${\bf PII}_0$ at $a$. 
\end{theorem} 

\begin{proof} 
We offer two based on standard uniqueness theorems, the one for ${\bf PII}_0$ and the other for ${\bf XX'}$. Let $s : = \sqrt S$.  (1) Suppose that $s$ were to satisfy ${\bf PII}_0$: as $s$ is non-negative, not only $s(a) = 0$ but also $\overset{\mdot}{s}(a) = 0$; now standard uniqueness forces $s = 0$ so that the zero $a$ is not isolated. (2) In fact, we claim that $s$ is not even twice-differentiable at $a$; for suppose it were. Again, $s(a) = 0$ and $\overset{\mdot}{s}(a) = 0$: as $S = s^2$ it follows that 
$$\overset{\mdot \mdot}{S}(a) = 2 s(a) \overset{\mdot \mdot}{s}(a) + 2 \overset{\mdot}{s}(a)^2 =0;$$ 
as $a$ is an isolated zero, this contradicts Theorem \ref{ne}. 
\end{proof} 

\medbreak 

Nevertheless, a solution $S$ to ${\bf XX}$ satisfying the hypotheses of this theorem {\it is} the square of a solution $s$ to ${\bf PII}_0$; it is simply the case that $s$ must change sign at $a$. 

\medbreak 

\begin{theorem} \label{root}
If $S$ is a solution to ${\bf XX}$ on $I$ and is strictly positive except for a zero at $a \in I$ then there exists a solution $s$ to ${\bf PII}_0$ on $I$ such that $S = s^2.$ 
\end{theorem} 

\begin{proof} 
Define $s$ on $I$ by 
 \begin{equation*}
    s(t)=
    \begin{cases}
      - \sqrt{S(t)} & \text{if}\ I \ni t \leqslant a, \\
      + \sqrt{S(t)} & \text{if} \ I \ni t \geqslant a.
    \end{cases}
  \end{equation*}
From Theorem \ref{square} it follows that $s$ satisfies ${\bf PII}_0$ on $I \setminus \{ a \}$; we must verify that $s$ is twice-differentiable at $a$ with $\overset{\mdot \mdot}{s}(a) = 0$. First of all, note that if $I \ni t \ne a$ then $\overset{\mdot}{s}(t) = \overset{\mdot}{S}(t)/2 s(t)$ whence 
$$\overset{\mdot}{s}(t)^2 = \frac{\overset{\mdot}{S}(t)^2}{4 s(t)^2} = \frac{\overset{\mdot}{S}(t)^2}{4 S(t)} = \frac{1}{2} \Big( \overset{\mdot \mdot}{S}(t) - 4 S(t)^2 - 2 t S(t) \Big)$$
and therefore 
$$\lim_{t \ra a} \overset{\mdot}{s}(t)^2 = \frac{1}{2} \overset{\mdot \mdot}{S}(a)$$
because $S(a) = 0$. Next, $\overset{\mdot}{S}(a) = 0$ while Theorem \ref{ne} informs us that $\overset{\mdot \mdot}{S}(a) > 0$; as a consequence, $\overset{\mdot}{S}(t)$ changes from strictly negative to strictly positive as $t$ increases through $a$. Thus $\overset{\mdot}{s} = \overset{\mdot}{S}/2 s$ is strictly positive on each side of $a$ and so the taking of square-roots yields 
$$\lim_{t \ra a} \overset{\mdot}{s}(t) = \sqrt{\frac{1}{2} \overset{\mdot \mdot}{S}(a)}.$$
An application of the mean value theorem now shows that the continuous function $s$ is continuously differentiable at $a$. Finally, as $s$ satisfies ${\bf PII}_0$ away from $a$ we deduce that 
$$\lim_{t \ra a} \overset{\mdot \mdot}{s}(t) = \lim_{t \ra a} \big(2 s(t)^3 + t s(t)\big) = 0$$
and a further application of the mean value theorem to the continuous function $\overset{\mdot}{s}$ permits us to conclude that $\overset{\mdot \mdot}{s}(a)$ exists and equals $0$. 
\end{proof} 

\medbreak 

\section*{Remarks} 

\medbreak 

Here, we consider matters of related interest, particularly concerning solutions to ${\bf XX}$ that are non-positive or change sign at an isolated zero. 

\medbreak 

\begin{theorem} \label{sigma}
If $S$ is a strictly negative solution to ${\bf XX}$ then $\sigma : = \sqrt{-S}$ is a solution to 
$$\overset{\mdot \mdot}{\sigma} = t \sigma - 2 \sigma^3.$$
\end{theorem} 

\begin{proof} 
Direct calculation from $S = - \sigma^2$ gives $\overset{\mdot}{S} = - 2 \sigma \overset{\mdot}{\sigma}$ and $\overset{\mdot \mdot}{S} = - 2 \sigma \overset{\mdot \mdot}{\sigma} - 2 \overset{\mdot}{\sigma}^2$; cancellation of $ - 2 \sigma$ following the invocation of ${\bf XX}$ concludes the argument. 
\end{proof} 

\medbreak 

Conversely, if $\sigma$ is a nowhere-zero solution to this differential equation, then it is readily checked that $S : = - \sigma^2$ is a strictly negative solution to ${\bf XX}$. To interpret the differential equation displayed in the theorem, notice that $\sigma$ satisfies this equation precisely when $s := i \sigma$ satisfies ${\bf PII}_0$. Of course, this interpretation is not entirely unexpected. 

\medbreak 

We leave the reader to contemplate the non-positive case, merely remarking that if $S$ is a solution to ${\bf XX}$ that has a single zero but is otherwise negative then $S = - \sigma^2$ for some (sign-changing) solution $\sigma$ to the differential equation of Theorem \ref{sigma}. 

\medbreak 

Our results on non-negative and non-positive solutions to ${\bf XX}$ are nicely complemented by the following result. 

\medbreak 

\begin{theorem} 
A solution to ${\bf XX}$ cannot change sign at an isolated zero. 
\end{theorem} 

\begin{proof} 
According to Theorem \ref{ne}, if the solution $S$ to ${\bf XX}$ has an isolated zero at $a$ then either $\overset{\mdot \mdot}{S}(a) > 0$ (in which case $S$ is strictly positive on each side of $a$) or $\overset{\mdot \mdot}{S}(a) < 0$ (in which case $S$ is strictly negative on each side of $a$). 
\end{proof} 

\medbreak 

By sharp contrast, a solution to ${\bf PII}_0$ {\it must} change sign at an isolated zero, as noted in the first proof of Theorem \ref{notroot}; of course, this circumstance also bears on Theorem \ref{square} and Theorem \ref{root}. 

\bigbreak

\begin{center} 
{\small R}{\footnotesize EFERENCES}
\end{center} 
\medbreak 

[Ince] E.L. Ince. {\it Ordinary Differential Equations}, Longman, Green and Company (1926); Dover Publications (1956). 

\medbreak

\end{document}